\pdfoutput=1

\documentclass[12pt]{article}
\usepackage{color}
\usepackage{dcolumn}
\usepackage{bm}
\usepackage{amsthm}
\usepackage{amssymb}
\usepackage{amsfonts}
\usepackage{mathrsfs}
\usepackage{stmaryrd}
\usepackage{geometry}
\usepackage{amsmath,comment}
\usepackage{enumerate}
\usepackage{setspace}
\usepackage{mathptmx}
\usepackage{verbatim}
\usepackage{tikz-cd}
\usepackage{cite}
\usepackage[hyphens]{url}
\usepackage[backref]{hyperref}
\usepackage{etoolbox}
\apptocmd{\sloppy}{\hbadness 10000\relax}{}{}

\hypersetup{
    breaklinks=true, 
    colorlinks=true,
    linkcolor=blue,
    filecolor=blue,
    citecolor=blue,      
    urlcolor=cyan,
}

{
    \newtheorem{theorem}{Theorem}[section]
    \newtheorem{lemma}[theorem]{Lemma}
    
    \newtheorem{proposition}[theorem]{Proposition}
    \newtheorem{remark}[theorem]{Remark}

    \theoremstyle{definition}
    \newtheorem{definition}[theorem]{Definition}
}

\newcommand{\PS}[1]{\llbracket #1 \rrbracket}
\newcommand{\LS}[1]{(\!( #1 )\!)}

\def\Q{\mathbb{Q}}

\def\F{\mathbb{F}}

\def\TF{\mathscr{F}}
\def\CN{\mathscr{N}}

\def\CP{\mathbb{CP}}
\def\Gal{{\rm Gal}}
\def\End{{\rm End}}
\def\Hom{{\rm Hom}}

\def\Spf{{\rm Spf}}
\def\MU{{\rm MU}}

\def\m{\mathfrak{m}}

\def\O{\mathcal{O}}

\begin{document}
    \title{A proof for Ando's theorem on norm-coherent coordinates via the Coleman norm operator}
    \author{Hongxiang Zhao}
    \date{}
    
    \maketitle
    
    \begin{abstract}
        Ando established an algebraic criterion for when a complex orientation for a Morava E-theory is an $H_\infty$-map. The criterion relates such an orientation to a specific property of the formal group associated to the E-theory, namely, a norm coherence condition on its coordinate. On the other hand, Coleman constructed a norm operator for interpolating division values in local fields, which depends on a Lubin--Tate formal group law. These formal group laws are important tools in explicit local class field theory. \par 
        In this article, we give a conceptual proof for Ando's theorem using the Coleman norm operator via the bridge of formal group laws between topology and arithmetic. 
    \end{abstract} 
    
    \tableofcontents

    \newpage

    \section{Introduction}
    The purpose of this article is to supply a number-theoretic, conceptual proof for an old, topological theorem of Matthew Ando's from his study of structured ring spectra and their multiplicative operations. We follow a suggestion by Charles Rezk and Yifei Zhu (cf. \cite[Remark 1.3]{Zhu20}). 
    \begin{itemize}
        \item We identify Ando's algebraic criterion for $H_\infty$ complex orientations of Morava E-theories as a coherence condition for certain division values in local fields with respect to a Coleman norm operator. 
        \item We then coherently construct these division values by an infinite iteration of the operator. 
    \end{itemize} 
    In a concrete way, the ingredients going into the proof hint towards a higher-algebraic theory for class fields (c.f., e.g., \cite{BSY22}). \par 
    Throughout, $k$ denotes a perfect field of characteristic $p > 0$, $R$ denotes a complete local ring with maximal ideal $\m$ and whose residue field $R/\m$ contains $k$. Given a formal group law $F \in R\PS{X,Y}$, we will write $X +_F Y := F(X,Y)$. \par 
    Morava E-theories $E_n$ are complex oriented cohomology theories in stable homotopy theory, whose coefficient ring $(E_n)_*$ in degree $0$ classifies deformations of a $1$-dimensional formal group law $G$ of height $n$ over $k$ to $R$. Each E-theory is represented by an $E_\infty$-spectrum \cite[Corollary 7.6]{GH04} and carries power operations associated to this multiplicative structure. Let $\MU$ be the complex cobordism theory. It is well-known that $\MU$ also admits power operations (\cite{TD68} and \cite[\S IV.2]{May77}). A complex orientation on $E_n$ is a map $\MU \to E_n$ of homotopy commutative ring spectra. Upon taking the connective cover $\MU\langle 0 \rangle$, this is equivalent to a coordinate (i.e., local uniformizer) of the universal deformation formal group associated with $E_n$, which in turn determines a formal group law \cite[Section 2]{Zhu20}. In the case that $G$ is a Honda formal group law over $\F_p$, Ando gave a criterion for when power operations on $\MU$ and $E_n$ are compatible along the map $\MU \to E_n$, i.e., when this map is $H_\infty$ \cite[Theorem 5]{And95}. The criterion is formulated in terms of the formal group law $F$ of the universal deformation associated to the coordinate. 
    \begin{theorem}[{{\cite[Theorem 4]{And95}}}] \label{Ando Theorem 4}
        Suppose $k = \F_p$ and $G$ is the Honda formal group law of height $n$ over $k$, so that $[p]_G(T) = T^{p^n}$, where $[p]_G$ is the $p$-series of $G$. Then in each $\star$-isomorphism class of lifts of $G$ to $\pi_0(E_n) \cong W(k)\PS{u_1,\cdots,u_{n-1}}$, there is a unique lift $F$ such that 
        \begin{equation*}
            [p]_F(T) = \prod_{\lambda} (T +_F \lambda)
        \end{equation*}
        where the product runs over all roots $\lambda$ of $[p]_F$. 
    \end{theorem}
    Such a coordinate is said to be \textbf{\em{norm-coherent}}, in that the right-hand side has the form of a norm map. The condition says that the canonical lift of Frobenius, which is just the map of multiplication by $p$ corresponding to the $p$-series in this case, coincides with the norm map. For a detailed discussion about the condition, see Section \ref{norm-coherent conditions} and \cite[Section 6]{Zhu20}. \par 
    On the number-theoretic side, an important tool invented in Lubin and Tate's explicit construction of the local Artin map in local class field theory is the family of Lubin--Tate formal group laws. Suppose a prime number $p$ is a uniformizer of a local field $K$, i.e., $K$ is an unramified extension of $\Q_p$. Then a Lubin--Tate formal group law reduces to a Honda formal group law over the residue field. In 1979, Coleman proved an interpolation theorem on division values in local fields, with applications to $p$-adic $L$-functions and modular units \cite[Theorem A]{Col79}. For that, he constructed a norm operator $\CN_F$ depending on a Lubin--Tate formal group law $F$ such that  
    \begin{equation*}
        \CN_F(g) \circ [p]_F(T) = \prod_{\{\lambda \colon [p]_F(\lambda) = 0\}} g(T +_F \lambda)
    \end{equation*}
    Rezk conjectured that the Coleman norm and Ando's algebraic criteria were closely related. Here, we prove Theorem \ref{Ando Theorem 4} via the norm operators $\CN_F$. \par 
    In more detail, the original definition of a norm operator restricts to the special case when $R$ is a complete discrete valuation ring with uniformizer $p$. We shall first give a description for Coleman's norm operators and prove Ando's theorem in this special case. Our proof will only depend on several properties of the norm operators and do not require $G$ to be a Honda formal group law. In view of this, we will generalize the definition of a norm operator to complete local domains in which $p \neq 0$, and show that the generalized norm operator satisfies the desired properties. In particular, $\pi_0(E_n)$ is a complete local domain with $p \neq 0$. The main result of this article is the following. 
    \begin{theorem} \label{Not precise goal}
        Suppose $G$ is a formal group law of finite height over $k$ and $R$ is a complete local domain with $p \neq 0$ whose residue field contains $k$. Then in each $\star$-isomorphism class of lifts of $G$ to $R$, there is a unique lift whose corresponding coordinate is norm-coherent. 
    \end{theorem}
    A more precise and explicit formulation of Theorem \ref{Not precise goal} will be given as Theorem \ref{Goal}. 
    \begin{remark} \label{Zhu and Ando}
        Zhu generalized the above theorem to apply to arbitrary complete local rings $R$ following the original proof by Ando \cite[Theorem 1.2]{Zhu20}. It would be interesting to have an alternative approach from number-theoretic constructions in this generality. 
    \end{remark}

    \section{Coleman norm operators from explicit local class field theory after Lubin and Tate} \label{Section 2}
    In this section, suppose that $k = \F_q$ with $q = p^n$. Suppose $K$ is a local field with integer ring $\O_K$, maximal ideal $\m$, and residue field $k$. Pick a uniformizer $\pi$ of $\O_K$ and let
    \begin{equation} \label{def of TF}
        \TF_\pi := \{\alpha(T) \in \O_K\llbracket T \rrbracket \colon \alpha(T) \equiv \pi T \mod{T^2},\ \ \alpha(T) \equiv T^q \mod{\pi}\}
    \end{equation}
    In explicit local class field theory, we have the following.  
    \begin{proposition}\label{LT FGL}
        For any $\alpha \in \TF_\pi$, there is a unique formal group law $F_\alpha$ over $\O_K$ such that $\alpha \in \End(F_\alpha)$. 
        \begin{proof}
            See \cite[I, 2.12]{Mil20}. 
        \end{proof}
    \end{proposition}
    \begin{proposition}\label{End LT FGL}
        \begin{itemize}
            \item [(a)] For any $\alpha,\beta \in \TF_\pi$ and $a \in \O_K$, there is a unique element $[a]_{\beta,\alpha}(T)$ in $T\O_K\llbracket T \rrbracket$ such that $[a]_{\beta,\alpha}(T) \equiv aT \mod{T^2}$ and $[a]_{\beta,\alpha} \in \Hom(F_\alpha,F_\beta)$. 
            \item [(b)] Moreover, the map $a \mapsto [a]_{\alpha,\alpha}(T)$ gives an isomorphism $\O_K \to \End(F_\alpha)$. In particular, $\alpha(T) = [\pi]_{\alpha,\alpha}(T)$. 
        \end{itemize}
        \begin{proof}
            See \cite[I, 2.14 and 2.17]{Mil20}. 
        \end{proof}
    \end{proposition}
    The formal group laws $F_\alpha$ characterized by Proposition \ref{LT FGL} are the \textbf{\em{Lubin--Tate formal group laws}}. Let $\alpha \in \TF_\pi$ and $\alpha = \tilde{\alpha}v$ for some unit $v \in \O_K\llbracket T \rrbracket$ and some polynomial $\tilde{\alpha} \in \O_K[T]$ by the Weierstrass preparation theorem. We then define a finite set 
    \begin{equation} \label{def of lambda_alpha,1}
        \Lambda_{\alpha,1} := \{\mbox{roots of }\tilde{\alpha}\mbox{ in a fixed algebraic closure of K}\}
    \end{equation}
    Suppose $\O_K\LS{T}$ is the ring of formal Laurent series with coefficients in $\O_K$. We assign the ``compact--open" topology to $\O_K\LS{T}$, i.e., a sequence $\{g_n\}$ converges to $g$ if and only if for any compact subset $A$ not containing $0$ in $\m$, and for each $\epsilon > 0$, there exists a positive integer $N = N(A,\epsilon)$ such that $|g_n(a) - g(a)| < \epsilon$ for all $a \in A$ and $n \geqslant N$. Given a Lubin--Tate formal group law $F_\alpha$, the Coleman norm operator is characterized by the following.  
    \begin{theorem} [{{\cite[Theorem 11 and Corollary 12]{Col79}}}] \label{Coleman Norm Operator}
        As notations above, there exists a unique function $\CN_{F_\alpha} \colon \O_K\LS{T} \to \O_K\LS{T}$ satisfying
        \begin{equation*}
            \CN_{F_\alpha}(g) \circ [p]_{F_\alpha}(T) = \prod_{\lambda \in \Lambda_{\alpha,1}} g(T +_{F_\alpha} \lambda)
        \end{equation*}
        for every $g \in \O_K\LS{T}$. Moreover, $\CN_{F_\alpha}$ is continuous and multiplicative. 
    \end{theorem}
    The norm operator has the following properties. 
    \begin{lemma} \label{Properties of CN}
        Let $i \geqslant 1$, $g \in 1 + \m^i\PS{T}$ and $h$ be a unit in $\O_K\LS{T}$. Then 
        \begin{enumerate}
            \item [(a)] $\CN_{F_\alpha}(g) \in 1 + \m^{i+1}\PS{T}$ and 
            \item [(b)] $\CN_{F_\alpha}^{i}(h)/\CN_{F_\alpha}^{i-1}(h) \in 1 + \m^i \PS{T}$, where $\CN_{F_\alpha}^i$ denotes $i$ iterations of applying the norm operator $\CN_{F_\alpha}$.  
        \end{enumerate}
        \begin{proof}
            See \cite[Lemma 13]{Col79}. Here, part (b) looks different from \cite[Lemma 13(b)]{Col79}, $\CN_{F_\alpha}^{i}(h)/\phi\CN_{F_\alpha}^{i-1}(h) \in 1 + \pi^i \O_K\PS{T}$, where $\phi \in \Gal(H/K)$ is the Frobenius map for a complete unramified extension $H/K$ and $\pi$ is the corresponding uniformizer. For our applications, we need only consider $K$ itself with $\phi$ the identity. 
        \end{proof}
    \end{lemma}
    As a consequence of Lemma \ref{Properties of CN}(b), the sequence $\{\CN_{F_\alpha}^i(h)\}$ converges in $\O_K\LS{T}$. Let $\CN_{F_\alpha}^\infty(h) := \displaystyle\lim_{i \to \infty} \CN_{F_\alpha}^i(h)$. In particular, by Lemma \ref{Properties of CN}(a), we have that 
    \begin{equation} \label{kernel of CN^infty}
        \CN_{F_\alpha}^\infty\bigl(1 + \m\PS{T}\bigr) = 1
    \end{equation}
    Since $\CN_{F_\alpha}$ is continuous, 
    \begin{equation*}
        \CN_{F_\alpha}\bigl(\CN_{F_\alpha}^\infty(h)\bigr) = \CN_{F_\alpha}\bigl(\displaystyle\lim_{i \to \infty} \CN_{F_\alpha}^i(h)\bigr) = \displaystyle\lim_{i \to \infty} \CN_{F_\alpha}\bigl(\CN_{F_\alpha}^i(h)\bigr) = \CN_{F_\alpha}^\infty (h)
    \end{equation*}
    Moreover, the operator$\CN_{F_\alpha}^\infty$ is multiplicative since $\CN_{F_\alpha}$ is.

    \section{Proof of Ando's theorem in a special case} \label{Special Proof}
    We will first prove Theorem \ref{Ando Theorem 4} (and \ref{Not precise goal}) in a special case using Coleman norm operators. \par 
    In this section, suppose that $K$ is an unramified extension of $\Q_p$ of degree $n$ and $G$ is the Honda formal group law over $k \cong \F_{p^n}$ of height $n$. Here, $p$ is a uniformizer of $K$. \par 
    In \eqref{def of TF}, choose the uniformizer $\pi = p$. Given any $\alpha \in \TF_\pi$, let $F_\alpha$ be the associated Lubin--Tate formal group law so that $[p]_{F_\alpha}(T) = [\pi]_{\alpha,\alpha}(T) = \alpha(T)$ by Proposition \ref{End LT FGL}(b). Thus, $F_\alpha$ is a lift of $G$ to $\O_K$. Conversely, every lift of $G$ to $\O_K$ has $p$-series in $\TF_\pi$, so by the uniqueness in Proposition \ref{LT FGL} it is a Lubin--Tate formal group law. 
    \begin{definition} [$\star$-isomorphisms]\label{naive star-iso}
        Two formal group laws $F,F^\prime$ over $\O_K$ are said to be \textbf{\em{$\star$-isomorphic}} if there is an isomorphism $u \colon F \to F^\prime$ such that $u$ restricts to the identity series to $k$. A more general definition of $\star$-isomorphism will be given in Section \ref{norm-coherent conditions}. 
    \end{definition}
    \begin{theorem} \label{Special Ando}
        With notations as above and in \eqref{def of lambda_alpha,1}, in each $\star$-isomorphism class of lifts of $G$ to $\O_K$, there is a unique formal group law $F = F_\alpha$, necessarily a Lubin--Tate formal group law for some $\alpha \in \TF_\pi$, satisfying 
        \begin{equation} \label{Ando's criterion}
            [p]_{F_\alpha}(T) = \prod_{\lambda \in \Lambda_{\alpha,1}}(T +_{F_\alpha} \lambda) 
        \end{equation}
    \end{theorem}
    To prove this special case of Ando's theorem, we proceed as follows. In terms of the norm operator, we see that a Lubin--Tate formal group law satisfies \eqref{Ando's criterion} if and only if 
    \begin{equation} \label{Ando's condition in term of CN before canceling [p]_F}
        [p]_{F_\alpha}(T) = \prod_{\lambda \in \Lambda_{\alpha,1}} (T +_{F_\alpha} \lambda) = \bigl(\CN_{F_\alpha}(T) \circ [p]_{F_\alpha}\bigr)(T)
    \end{equation} 
    Since $p$ is invertible in $K$, $[p]_{F_\alpha}$ has a composition inverse in $K\PS{T}$. We can thus cancel the term from both sides above, so that condition \eqref{Ando's criterion} is equivalent to 
    \begin{equation} \label{Walker's conclusion}
        \CN_{F_\alpha}(T) = T
    \end{equation}
    Begin with any lift $F_\alpha$ of $G$ to $\O_K$. Let $u \in T + \pi T \O_K\PS{T} = T + T\m\PS{T}$. Then there is an element $\alpha_u \in \TF_\pi$ such that $u \circ F_\alpha \circ u^{-1} = F_{\alpha_u}$. Indeed, since $\alpha = [p]_{F_\alpha}$ and $\alpha_u = [p]_{F_{\alpha_u}}$, we have $\alpha_u = u \circ \alpha \circ u^{-1}$. Clearly $F_\alpha$ and $F_{\alpha_u}$ are $\star$-isomorphic. In order to show that there is a unique $u \in T + T\m\PS{T}$ such that $F_{\alpha_u}$ satisfies \eqref{Ando's criterion}, we are reduced to show that there is a unique $u \in T + T\m\PS{T}$ such that 
    \begin{equation} \label{condition for CN_F_alpha_u}
        \CN_{F_{\alpha_u}}(T) = T
    \end{equation}
    Note that $u$ induces a bijection from $\Lambda_{\alpha,1}$ to $\Lambda_{\alpha_u,1}$. By definition,  
    \begin{equation*}
        \bigl(\CN_{F_{\alpha_u}}(T) \circ [p]_{F_{\alpha_u}}\bigr)(T) = \prod_{\lambda \in \Lambda_{\alpha_u,1}} (T +_{F_
        {\alpha_u}} \lambda)
    \end{equation*}
    We rewrite this identity as follows:  
    \begin{align*}
        \bigl(\CN_{F_{\alpha_u}}(T) \circ u \circ [p]_{F_\alpha} \circ u^{-1}\bigr)(T) &= \prod_{\lambda \in \Lambda_{\alpha,1}} \bigl(T +_{F_{\alpha_u}} u(\lambda)\bigr) \\ 
        &= \prod_{\lambda \in \Lambda_{\alpha,1}} F_{\alpha_u}\biggl(u\bigl(u^{-1}(T)\bigr),u(\lambda)\biggr) \\ 
        &= \prod_{\lambda \in \Lambda_{\alpha,1}} u \circ F_\alpha\bigl(u^{-1}(T),\lambda\bigr) \\ 
        &= \prod_{\lambda \in \Lambda_{\alpha,1}} u \circ \bigl(u^{-1}(T) +_{F_\alpha} \lambda\bigr) \\ 
        &= \bigl(\CN_{F_\alpha}(u) \circ [p]_{F_\alpha}\bigr) \bigl(u^{-1}(T)\bigr)
    \end{align*}
    By canceling $[p]_{F_\alpha} \circ u^{-1}$ from both sides, we obtain $\bigl(\CN_{F_{\alpha_u}}(T) \circ u\bigr)(T) = \CN_{F_\alpha}(u)(T)$. Therefore, \eqref{condition for CN_F_alpha_u} is equivalent to 
    \begin{equation*}
        \CN_{F_\alpha}(u) = u
    \end{equation*}
    Consequently, it remains to show the following. 
    \begin{proposition}
        Given any $\alpha \in \mathcal{F}_\pi$, there is a unique $u \in T + T\m\PS{T}$, such that $\CN_{F_\alpha}(u) = u$. 
    \end{proposition}
    \begin{proof} Existence: By Lemma \ref{Properties of CN}(b), let $h_i := \CN_{F_\alpha}^i(T)/\CN_{F_\alpha}^{i-1}(T) \in 1 + \m^i\PS{T}$. Then we have $\CN_{F_\alpha}^\infty(T) = Th_1h_2\cdots$. It is easy to see that $h_1h_2\cdots \in 1 + \m\PS{T}$, so $\CN_{F_\alpha}^\infty(T) \in T + T\m\PS{T}$. Therefore, $u = \CN_{F_\alpha}^\infty(T)$ satisfies the condition. \par 
        Uniqueness: If $\CN_{F_\alpha}(u) = u$, then $\CN_{F_\alpha}^i(u) = u$ for each $i$. Thus, $\CN_{F_\alpha}^\infty (u) = u$ after taking the limit. Since $u \in T + T\m\PS{T}$, there is $\tilde{u} \in 1 + \m\PS{T}$ such that $u = T\tilde{u}$. Then, in view of \eqref{kernel of CN^infty}, 
        \begin{equation*}
            u = \CN_{F_\alpha}^\infty(u) = \CN_{F_\alpha}^\infty(T) \CN_{F_\alpha}^\infty(\tilde{u}) = \CN_{F_\alpha}^\infty(T)
        \end{equation*}
    \end{proof}
    \begin{remark}
        We can interpret the equality $\CN_{F_\alpha}(u) = u$ as a condition of norm coherence in the context of \cite{Col79}. To be precise, suppose $\Lambda_{\alpha,n}$ is the set of roots of $[p^n]_{F_\alpha}$ in the fixed algebraic closure of $K$. Let $K_{\pi,n} := K(\Lambda_{\alpha,n})$ and $N_{n+1,n} := N_{K_{\pi,n+1}/K_{\pi,n}}$ be the norm map. It can be shown that $\Lambda_{\alpha,n} \cong \O_K/\m^n$ \cite[I, 3.4]{Mil20}. Suppose $v_n$ is a generator of $\Lambda_{\alpha,n}$ as an $\O_K$-module for each $n$ such that $[p]_{F_\alpha}(v_{n+1}) = v_n$. We then have 
        \begin{equation*}
            \CN_{F_\alpha}(u)(v_n) = N_{n+1,n}\bigl(u(v_{n+1})\bigr)
        \end{equation*}
        by \cite[Corollary 12(ii)]{Col79}. Thus, $\CN_{F_\alpha}(u) = u$ is equivalent to saying that 
        \begin{equation*}
            u(v_n) = N_{n+1,n}\bigl(u(v_{n+1})\bigr)
        \end{equation*}
        i.e., $u$ transforms the sequence $\{v_n\}$ to a sequence compatible with the norm maps.  
        \begin{remark}
            Let $\mathscr{M}_{\infty,\alpha} = \{g \in \O_K\LS{T}^* \colon \CN_{F_\alpha}(g) = g\}$ be the subset in $\O_K\LS{T}^*$ consisting of norm-coherent series in the sense above. Then the uniqueness of $u$ follows from an exact sequence of groups. 
            \begin{equation*}
                1 \to 1 + \m\PS{T} \to \O_K\LS{T}^{\times} \stackrel{\CN_{F_\alpha}^\infty}{\longrightarrow} \mathscr{M}_{\infty,\alpha} \to 1
            \end{equation*}
            as in \cite[Proposition 14]{Col79}. 
        \end{remark}
    \end{remark}

    \section{Norm coherence condition} \label{norm-coherent conditions}
    In a more general case of Theorem \ref{Special Ando}, the left-hand side of \eqref{Ando's criterion} should not simply be $[p]_{F_\alpha}(T)$. Indeed, it should be a canonical lift of the relative Frobenius map. To define this lift, we have to first recall some notions in formal groups and deformations of formal group laws. \par 
    Suppose $\mathcal{F}$ is a formal group over $R$ and $\mathcal{G}$ is a formal group over $k$ of finite height $n$. Let $F$ and $G$ be the respective formal group laws associated to chosen coordinates of these formal groups. Let $X$ be a chosen coordinate on $\mathcal{F}$. 
    \begin{definition}[Quotients of formal groups] \label{def of quotients of fg}
        Suppose $\mathcal{D}$ is a subgroup of $\mathcal{F}$ of degree $p^r$ defined over a complete local ring $S$ containing $R$. Then we define the \textbf{\em{quotient group}} $\mathcal{F}/\mathcal{D}$ over $S$ as follows. Let $m \colon \mathcal{F} \times \mathcal{D} \to \mathcal{F}$ be the multiplication map and $pr_1 \colon \mathcal{F} \times \mathcal{D} \to \mathcal{F}$ be the projection onto the first factor. The coordinate ring of $\mathcal{F}/\mathcal{D}$ is defined by the equalizer 
        \begin{center}
            \begin{tikzcd}
                {\O_{\mathcal{F}/\mathcal{D}}} & {\O_{\mathcal{F}}} & {\O_{\mathcal{F} \times \mathcal{D}}}
                \arrow["{pr_1^*}"', shift right=1, from=1-2, to=1-3]
                \arrow["{m^*}", shift left=1, from=1-2, to=1-3]
                \arrow["{f_\mathcal{D}^*}", from=1-1, to=1-2]
            \end{tikzcd}
        \end{center}
        It can be shown that $\mathcal{F}/\mathcal{D}$ is a formal group over $S$. In addition, viewing $\O_\mathcal{F}$ as an $\O_{\mathcal{F}/\mathcal{D}}$-module through $f_\mathcal{D}^*$, we obtain $X_\mathcal{D} = \mbox{Norm}_{f_\mathcal{D}^*}(X)$ as a coordinate of $\mathcal{F}/\mathcal{D}$ \cite[Theorem 19]{Str97}. 
    \end{definition}
    If $\mathcal{D}(S)$ contains exactly $p^r$ elements, then 
    \begin{equation} \label{f_D(X_D) = prod}
        f_\mathcal{D}^*(X_\mathcal{D}) = \prod_{P \in \mathcal{D}(S)}\bigl(X +_F X(P)\bigr)
    \end{equation}
    by direct calculation. Thus, reducing \eqref{f_D(X_D) = prod} to the residue field of $R$ we have 
    \begin{equation} \label{f restricts to Frob} 
        f_\mathcal{D}^*(X_\mathcal{D}) \equiv X^{p^r} \mod{\m}
    \end{equation}
    \textbf{Notation. }For simplicity, from now on, we will not distinguish an isogeny between formal groups and the power series to which it corresponds as a map between coordinate rings. For instance, we will simply write
    \begin{equation*}
        f_\mathcal{D}(T) = \prod_{P \in \mathcal{D}(S)}\bigl(T +_F X(P)\bigr)
    \end{equation*}
    as an isogeny between formal group laws. \par 
    The following definition generalizes Definition \ref{naive star-iso}. 
    \begin{definition}[Deformations of formal group laws and $\star$-isomorphisms]
        Let $\pi \colon R \to R/\m$ be the natural projection. A \textbf{\em{deformation of $G$ to $R$}} is a triple $(F,i,\eta)$, where $F$ is a formal group law over $R$, $i \colon k \to R/\m$ is an inclusion and $\eta \colon \pi^*(F) \to i^*(G)$ is an isomorphism. Here $\pi$ and $i$ act on each coefficient. \par 
        Suppose $(F,i,\eta)$ and $(F^\prime,i^\prime,\eta^\prime)$ are two deformations of $G$ to $R$ such that $i = i^\prime$. Then we say $(F,i,\eta)$ and $(F^\prime,i^\prime,\eta^\prime)$ are \textbf{\em{$\star$-isomorphic}} if there is an isomorphism $\psi \colon F \to F^\prime$ of formal group laws such that $\eta^\prime \circ \pi^*(\psi) = \eta$. \par 
        Furthermore, if $\eta = \eta^\prime$ and $\pi^*(\psi) = \rm{id}$, we call $\psi \colon F \to F^\prime$ a \textbf{\em{$\star$-isomorphism}} as well. 
    \end{definition}
    \begin{definition}[Deformations of Frobenius]
        Suppose $\Phi$ is the relative Frobenius map on $G$ over $k$ and $\sigma$ is the absolute Frobenius map over $k$. Suppose $(F,i,\eta)$ and $(F^\prime,i^\prime,\eta^\prime)$ are two deformations of $G$ to $R$. An isogeny $\psi \colon F \to F^\prime$ of degree $p^r$ is a \textbf{\em{deformation of Frobenius}} if $i^\prime = i \circ \sigma^r$, and $\eta^\prime \circ \pi^*(\psi) = i^*(\Phi^r) \circ \eta$, i.e., the following diagram commutes. 
        \begin{center}
            \begin{tikzcd}
                {\mathcal{F}} & {\pi^*(\mathcal{F})} & {i^*(\mathcal{G})} \\
                && {i^*(\mathcal{G}^{(p^r)})} \\
                {\mathcal{F^\prime}} & {\pi^*(\mathcal{F^\prime})} & {{i^\prime}^*(\mathcal{G})}
                \arrow[from=1-2, to=1-1]
                \arrow["\psi"', from=1-1, to=3-1]
                \arrow["{\pi^*(\psi)}"', from=1-2, to=3-2]
                \arrow[from=3-2, to=3-1]
                \arrow["{i^*(\Phi^r)}", from=1-3, to=2-3]
                \arrow["\eta", from=1-2, to=1-3]
                \arrow["{\eta^\prime}"', from=3-2, to=3-3]
                \arrow[Rightarrow, no head, from=2-3, to=3-3]
            \end{tikzcd}
        \end{center} \par 
        Two deformations of Frobenius $(F_1,i_1,\eta_1) \to (F_1^\prime,i_1^\prime,\eta_1^\prime)$ and $(F_2,i_2,\eta_2) \to (F_2^\prime,i_2^\prime,\eta_2^\prime)$ are said to be \textbf{\em{isomorphic}} if $(F_1,i_1,\eta_1), (F_2,i_2,\eta_2)$ are $\star$-isomorphic and $(F_1^\prime,i_1^\prime,\eta_1^\prime), (F_2^\prime,i_2^\prime,\eta_2^\prime)$ are $\star$-isomorphic. 
    \end{definition}
    The following theorem classifies deformations of Frobenius. 
    \begin{theorem} [{{cf. \cite[Theorem 42]{Str97}}}] \label{classification of def of Frob}
        There is a universal deformation $(F_{\rm univ},{\rm id},{\rm id})$ of $G$ to the Lubin--Tate ring $W(k)\PS{u_1,\cdots,u_{n-1}}$, in the following sense. For each $r \geqslant 0$, there is a complete local ring $A^r$ such that 
        \begin{equation*}
            \{\mbox{deformations } (F,i,\eta) \to (F^\prime,i^\prime,\eta^\prime) \mbox{ of } \Phi^r \mbox{ to } T\}/\mbox{isomorphisms} \cong \Hom(A^r,T)
        \end{equation*}
        Moreover, we have $A^0 = W(k)\PS{u_1,\cdots,u_{n-1}}$ and $A^r$ is a bimodule over $A^0$ with structure maps $s^r,t^r \colon A^0 \to A^r$, which are local homomorphisms. \par 
        The isomorphism is given as follows: for any deformation $(F,i,\eta) \to (F^\prime,i^\prime,\eta^\prime)$ of \ $\Phi^r$ to $T$, there is a unique local homomorphism $\rho^r \colon A^r \to T$ such that $\rho^r \circ s^r$ and $\rho^r \circ t^r$ restrict to $i$ and $i^\prime$ on the residue fields respectively and that there are unique $\star$-isomorphisms $(F,i,\eta) \to ({\rho^r}^*{s^r}^*F_{\rm univ},i,\rm{id})$ and $(F^\prime,i^\prime,\eta^\prime) \to ({\rho^r}^*{t^r}^*F_{\rm univ}, i^\prime, {\rm id})$.
    \end{theorem}
    \begin{remark} \label{rmk: eta to id}
        According to \cite[Remark 6.2]{Zhu20}, given any deformation $(F,i,\eta)$ of $G$ to $R$, there exists a unique deformation $(\tilde{F},i,\rm{id})$ given by \cite[Theorem 3.1]{LT66} such that the two deformations are $\star$-isomorphic. Thus, we will assume $\eta = \rm{id}$ in the following.
   \end{remark} 
    Suppose from now on that $(F,i,\rm{id})$ is a deformation of $G$ to $R$, where $R$ is a complete local domain with $p \neq 0$. \par 
    Since $[p]_F \equiv [p]_G \mod{\m}$ and $G$ has height $n < \infty$, not all coefficients of $[p]_F$ are in $\m$. By the Weierstrass preparation theorem, there is a unit $v$ in $R\PS{T}$ and a monic polynomial $\beta$ of degree $p^n$, such that $[p]_F = v \cdot \beta$. Note that roots of $[p]_F$ are the same as the roots of $\beta$. Let
    \begin{equation} \label{def of lambda_p,1}
        \Lambda_{[p]_F,1} := \{\mbox{roots of $[p]_F$ in a larger ring $S$ obtained from $R$ by adjoining roots of $\beta$}\}
    \end{equation}
    Since $p \neq 0$ in $R$, $0$ is a simple root of $[p]_F$. For any $\lambda \in \Lambda_{[p]_F,1}$, $[p]_F\bigl(T -_F \lambda\bigr) = [p]_F(T)$. Therefore, $\lambda$ is also a simple root of $[p]_F$. We conclude that the set $\Lambda_{[p]_F,1}$ has exactly $p^n$ elements. 
   
    Let $\mathcal{D} := \mathcal{F}[p]$ be the subgroup of $\mathcal{F}$ of $p$-torsions defined over $S$. Then the quotient isogeny $f_p \colon \mathcal{F} \to \mathcal{F}/\mathcal{F}[p]$ is given by 
    \begin{equation*}
        f_p(T) = \prod_{\lambda \in \Lambda_{[p]_F,1}}\bigl(T +_F \lambda\bigr)
    \end{equation*}
    Note that $f_p(T)$ is invariant under the action of $\mbox{Aut}(S/R)$, so $f_p(T) \in R\PS{T}$. As a consequence, $\mathcal{F}/\mathcal{F}[p]$ can be defined over $R$. Thus, $f_p$ is a deformation of Frobenius between $(F,i,\rm{id})$ and $(F/F[p],i \circ \sigma^n,\rm{id})$ in view of \eqref{f restricts to Frob}. By Theorem \ref{classification of def of Frob}, there exists a unique local homomorphism $\rho^n \colon A^n \to R$ with a unique $\star$-isomorphism $g_p \colon F/F[p] \to {\rho^n}^*{t^n}^*F_{\rm univ}$. 
    \begin{definition}
        Define 
        \begin{equation*}
            l_p := g_p \circ f_p \colon F \to {\rho^n}^*{t^n}^*F_{\rm univ}
        \end{equation*}
    \end{definition}
    \begin{remark} \label{Classification of l_p}
        According to \cite[Remark 6.7]{Zhu20}, the isogeny $l_p$ is an isogeny of formal group laws over $R$ characterized by the following properties. 
        \begin{itemize}
            \item [(a)] It is an isogeny from $F$ to ${\rho^n}^*{t^n}^*F_{\rm univ}$. 
            \item [(b)] The kernel of $l_p$ is the same as that of $[p]_F$. 
            \item [(c)] The reduction of $l_p$ to the residue field is the $p^n$-power relative Frobenius $T^{p^n}$, i.e., $l_p(T) \equiv T^{p^n} \mod{\m}$. 
        \end{itemize}
    \end{remark}
    To proceed to the general case, we formulate the norm coherence condition of Theorem \ref{Not precise goal} precisely as follows. 
    \begin{theorem} \label{Goal}
        Suppose $R$ is a complete local domain with $p \neq 0$ and residue field containing $k$. In each $\star$-isomorphism class of deformations of $G$ to $R$, there is a unique element $(F,i,{\rm id})$ such that 
        \begin{equation} \label{Generalized Ando's criterion}
            l_p(T) = f_p(T) = \prod_{\lambda \in \Lambda_{[p]_F,1}}\bigl(T +_F \lambda\bigr)
        \end{equation}
        i.e., $g_p(T) = T$. 
    \end{theorem}
    \begin{remark}
        According to the proof for \cite[Proposition 7.1]{Zhu20}, the coordinate given by Theorem \ref{Goal} is norm-coherent with respect to any finite subgroup $D \subset \mathcal{F}$ in the sense of \cite[Definition 6.21]{Zhu20}. 
    \end{remark}
    \begin{remark}
        Back to the topological side, an orientation on the Morava E-theory $E_n$ corresponds to a coordinate on the formal group $\Spf\bigl(\pi_0\bigl(E_n^{\CP_+^\infty}\bigr)\bigr)$ \cite[Example 2.53]{AHS01}, so it induces a deformation $(F,i,\eta)$ of $G$ to $\pi_0\bigl(E_n^{\CP_+^\infty}\bigr)$. A change of orientation on $E_n$ will induce a $\star$-isomorphism between the induced deformations \cite[Example 4.9]{Zhu20}. One should be aware that the $\star$-isomorphism here may change the $\eta$-component. Recall that we set $\eta = \rm{id}$ in Remark \ref{rmk: eta to id}. Thus, Theorem \ref{Goal} implies that there is a unique $H_\infty$-orientation $\MU\langle 0 \rangle \to E_n$ such that the induced deformation has $\eta = \rm{id}$ and satisfies $\eqref{Generalized Ando's criterion}$. \par 
        In general, we say a deformation $(F,i,\eta)$ (also the orientation inducing this deformation) is \textbf{\em norm-coherent} if the component $\tilde{F}$ in $(\tilde{F},i,{\rm id})$ is norm-coherent in the sense of \eqref{Generalized Ando's criterion}, where the latter deformation is given by Remark \ref{rmk: eta to id} (cf. \cite[Definition 6.21]{Zhu20}). Therefore, there is only a family of norm-coherent orientations on $E_n$, which are $\star$-isomorphic to the deformation $(F,i,{\rm id})$ given by Theorem \ref{Goal}. 
    \end{remark}

    \section{Generalization of the norm operators}
    In this section, we aim to prove Theorem \ref{Goal} (and hence Theorem \ref{Not precise goal}) following the proof of the special case in Section \ref{Special Proof}. Observe that our earlier proof actually only requires $\O_K$ to be a complete local domain such that $[p]_{F_\alpha}$ is right-cancellative and the definition and properties of the Coleman norm operator over $\O_K$ in Section \ref{Section 2}. Therefore, we need only show that $l_p$ is right-cancellative and generalize Theorem \ref{Coleman Norm Operator} and Lemma \ref{Properties of CN} to the case of a complete local domain $R$ such that $p \neq 0$ and residue field contains $k$ substituting $[p]_F$ by $l_p$. The proof then applies mutatis mutandis. 
    \begin{remark}
        In the general case, we do not require $G$ to be a Honda formal group law and $F$ to be a Lubin-Tate formal group law as in Section \ref{Special Proof}, since the property that $[p]_F \equiv T^{p^n} \mod{\m}$ has been replaced by Remark \ref{Classification of l_p}(c). 
    \end{remark}
    \begin{lemma} \label{l_p can be cancaled from right}
        The power series $l_p$ is right-cancellative, i.e., if there are $g,h \in R\PS{T}$ such that $g \circ l_p = h \circ l_p$, we have $g = h$. 
        \begin{proof}
            We may assume $h = 0$, so that $g\bigl(l_p(T)\bigr) = 0$. Since $R$ is complete with respect to $\m$, we need only prove by induction on $i$ that $g \equiv 0 \mod{\m^i}$ for each $i \geqslant 0$. The statement is vacuous for $i = 0$. Suppose we have proven that $g \equiv 0 \mod{\m^i}$. Since $l_p(T) \equiv T^{p^n} \mod{\m}$ by Remark \ref{Classification of l_p}(c), $g(T^{p^n}) \equiv 0 \mod{\m^{i+1}}$, and hence $g \equiv 0 \mod{\m^{i+1}}$. 
        \end{proof}
    \end{lemma}
    The following proofs essentially follow those in \cite{Col79}. 
    \begin{lemma}[{{cf. \cite[Lemma 3]{Col79}}}] \label{Lemma for CN}
        If $g \in R\PS{T}$ and $g(T +_F \lambda) = g(T)$ for all $\lambda \in \Lambda_{[p]_F,1}$, then there is a unique $h \in R\PS{T}$ such that $h \circ l_p = g$. 
        \begin{proof}
            The uniqueness follows from Lemma \ref{l_p can be cancaled from right}. \par 
            For the existence, we inductively construct formal power series $g_m$ for each $m \geqslant 0$. Let $g_0 = g$. Suppose that we have constructed $a_i \in R$ for $0 \leqslant i \leqslant m - 1$ such that 
            \begin{equation*}
                g - \sum_{i = 0}^{m - 1} a_il_p^i = l_p^m \cdot g_m
            \end{equation*}
            for some $g_m \in R\PS{T}$. Note that by assumption $g(T +_F \lambda) = g(T)$ and by Remark \ref{Classification of l_p}(b) $l_p(T +_F \lambda) = l_p(T)$. We thus have $g_m(T +_F \lambda) = g_m(T)$. In particular, $\bigl(g_m -g_m(0)\bigr)(\lambda) = 0$ for all $\lambda \in \Lambda_{[p]_F,1}$. By the Euclidean algorithm for power series, there exists elements $g_{m+1} \in R\PS{T}$ with $r_m \in R[T]$ such that $g_m - g_m(0) = l_p \cdot g_{m+1} + r_m$ and $\deg(r_m) < p^n$. Then $r_m$ vanishes on $\Lambda_{[p]_F,1}$. Since we have deduced in the paragraph after \eqref{def of lambda_p,1} that $\Lambda_{[p]_F,1}$ has $p^n$ elements, we have $r_m = 0$. Let $a_m = g_m(0)$. This finishes the inductive step. Thus, we obtain 
            \begin{equation*}
                g - \sum_{i = 0}^\infty a_il_p^i \in \bigcap_{i = 0}^\infty l_p^i R\PS{T} = 0
            \end{equation*}
            Then $h(T) := \sum_{i = 0}^\infty a_i T^i$ is the desired element. 
        \end{proof}
    \end{lemma}
    Now we also endow $R\PS{T}$ with the compact--open topology similar to that on $\O_K\PS{T}$ as in Section \ref{Section 2}. Here $R$ has the $\m$-adic topology. 
    \begin{theorem} [{{cf. Theorem \ref{Coleman Norm Operator}}}] \label{generalized CN}
        There is a unique operator $\CN_F \colon R\PS{T} \to R\PS{T}$ such that for any $g \in R\PS{T}$, 
        \begin{equation*}
            \CN_F(g) \circ l_p (T) = \prod_{\lambda \in \Lambda_{[p]_F,1}}g(T +_F \lambda)
        \end{equation*}
        Moreover, $\CN_F$ is multiplicative and continuous. 
        \begin{proof}
            Note that the right-hand side of the above identity satisfies the condition of $\Lambda_{[p]_F,1}$-invariance from Lemma \ref{Lemma for CN}. Thus, there is a unique $\CN_F$ satisfying the formula. \par 
            Given any $g,h \in R\PS{T}$, 
            \begin{align*}
                \begin{split}
                    \CN_F(gh) \circ l_p(T) &= \prod_{\lambda \in \Lambda_{[p]_F,1}}gh(T +_F \lambda) \\ 
                    &= \bigl(\CN_F(g) \circ l_p(T)\bigr) \cdot \bigl(\CN_F(h) \circ l_p(T)\bigr) \\ 
                    &= \bigl(\CN_F(g) \cdot \CN_F(h)\bigr) \circ l_p(T) 
                \end{split}
            \end{align*}
            By Lemma \ref{l_p can be cancaled from right}, we then obtain $\CN_F(gh) = \CN_F(g) \cdot \CN_F(h)$. \par 
            Suppose $\{g_n\}$ converges to $g$. Then
            \begin{align*}
                \begin{split}
                    \bigl(\lim \CN_F(g_n)\bigr) \circ l_p(T) &= \lim \bigl(\CN_F(g_n) \circ l_p\bigr)(T) = \lim \prod_{\lambda \in \Lambda_{[p]_F,1}} g_n(T +_F \lambda) \\ 
                    &= \prod_{\lambda \in \Lambda_{[p]_F,1}} g(T +_F \lambda) = \CN_F(g) \circ l_p(T)
                \end{split}
            \end{align*}
            Again, it follows from Lemma \ref{l_p can be cancaled from right} that $\lim \CN_F(g_n) = \CN_F(g)$. 
        \end{proof}
    \end{theorem}
    Note that in the proof in Section \ref{Special Proof} we only took limits of $\CN_{F_\alpha}$ applied to elements in $1 + \m\PS{T}$ and to $T$. Thus, it remains to show the following for $\CN_F$. 
    \begin{lemma} [{{cf. Lemma \ref{Properties of CN}}}] \label{generalized properties of CN}
        Let $g \in 1 + \m^i\PS{T}$ and $i \geqslant 1$. Then 
        \begin{enumerate}
            \item [(a)] $\CN_F(g) \in 1 + \m^{i+1}\PS{T}$ and 
            \item [(b)] $\CN_F^i(T)/\CN_F^{i-1}(T) \in 1 + \m^i\PS{T}$. 
        \end{enumerate}
        \begin{proof}
            \begin{enumerate}
                \item [(a)] By definition, $\CN_F(g) \circ l_p(T) = \prod_{\lambda \in \Lambda_{[p]_F,1}} g(T +_F \lambda)$. Suppose that $g(T) = 1 + \sum_{j=0}^\infty b_j T^j$, where $b_j \in \m^i$. Since $i \geqslant 1$, terms divided by $b_{j_1}b_{j_2}$ for some $j_1,j_2$ must lie in $\m^{i+1}$. Therefore, modulo $\m^{i+1}$, 
                \begin{align*}
                    \begin{split}
                        \CN_F(g) \circ l_p(T) &\equiv 1 + \sum_{\lambda \in \Lambda_{[p]_F,1}} \sum_{j=0}^\infty b_j(T +_F \lambda)^j \\ 
                        &= 1 + \sum_{j=0}^\infty\sum_{\lambda \in \Lambda_{[p]_F,1}} b_j(T +_F \lambda)^j \\ 
                        &= 1 + \sum_{j=0}^\infty b_j\bigl(p^nT^j + \sum_{k=0}^\infty p_k(\Lambda_{[p]_F,1})T^k\bigr) 
                    \end{split}
                \end{align*}
                where each $p_k(\Lambda_{[p]_F,1})$ is a symmetric function on $\lambda \in \Lambda_{[p]_F,1}$. Recall from Section \ref{norm-coherent conditions} that $\beta$ is a polynomial of degree $p^n$ dividing $[p]_F$ from the Weierstrass preparation theorem, and $\Lambda_{[p]_F,1}$ is the set of roots of $\beta$. Thus, $\beta(T) \equiv T^{p^n} \mod{\m}$. It follows that $p_k(\Lambda_{[p]_F,1}) \equiv 0 \mod{\m}$, since $p_k(\Lambda_{[p]_F,1})$ is a polynomial in the non-leading coefficients of $\beta$ (i.e., those elementary symmetric functions). Therefore, 
                \begin{equation*}
                    \CN_F(g) \circ l_p(T) \equiv 1 \mod{\m^{i+1}}
                \end{equation*}
                Next we prove by induction on $j$ that if $h \in R\PS{T}$ and $h \circ l_p \in \m^j\PS{T}$, then $h \in \m^j\PS{T}$ (here $j \geqslant 0$). Setting $h = \CN_F(g) - 1$ completes the proof of part (a). The case is vacuous when $j = 0$. Suppose $j \geqslant 1$ and the statement holds for $j - 1$. By the induction hypothesis, $h \in \m^{j-1}\PS{T}$. Since $l_p(T) \equiv T^{p^n} \mod{\m}$, we have 
                \begin{equation*}
                    h \circ l_p(T) \equiv h(T^{p^n}) \mod{\m^j}
                \end{equation*}
                Since $h \circ l_p \in \m^j\PS{T}$, we must have $h \in \m^j\PS{T}$. 
                \item [(b)] By part (a) and the multiplicativity of $\CN_F$ from Theorem \ref{generalized CN}, we need only show the case when $i = 1$. Since $l_p(T) \equiv T^{p^n} \mod{\m}$, 
                \begin{equation*}
                    \CN_F(T)(T^{p^n}) \equiv \CN_F(T) \circ l_p(T) = \prod_{\lambda \in \Lambda_{[p]_F,1}} (T +_F \lambda) \mod{\m}
                \end{equation*}
                By an argument with symmetric functions similar to that for part (a), we obtain that $\prod_{\lambda \in \Lambda_{[p]_F,1}} (T +_F \lambda) \equiv T^{p^n} \mod{\m}$. Hence, $\CN_F(T) \equiv T \mod{\m}$. Therefore, we have that $\CN_F(T)/T \equiv 1 \mod{T^{-1}\m\PS{T}}$. It remains to show that $T$ divides $\CN_F(T)$ in $R\PS{T}$, or $\CN_F(T)(0) = 0$. Since $0 \in \Lambda_{[p]_F,1}$, 
                \begin{equation*}
                    \CN_F(T)(0) = \CN_F(T) \circ l_p (0) = \prod_{\lambda \in \Lambda_{[p]_F,1}}\lambda = 0 
                \end{equation*}
            \end{enumerate}
        \end{proof}
    \end{lemma}
    \begin{remark}
        Comparing \eqref{Ando's condition in term of CN before canceling [p]_F} from the proof in Section \ref{Special Proof}, we see that $f_p$ as defined in Section \ref{norm-coherent conditions} now factors as $\CN_F(T) \circ l_p$. Recall that $l_p = g_p \circ f_p$ by definition. By Lemma \ref{l_p can be cancaled from right}, $l_p$ is right-cancellative with respect to composition, and hence so does $f_p$. Therefore, $\CN_F(T)$ turns out to be the inverse to $g_p$, i.e., it is a (necessarily unique) $\star$-isomorphism ${\rho^n}^*{t^n}^*F_{\rm univ} \to F/F[p]$. This agrees with the construction of the coordinate associated with $F/F[p]$ in Definition \ref{def of quotients of fg}. \par 
        In view of this, the proof in Section \ref{Special Proof} is actually similar to the one for \cite[Theorem 2.6.4]{And95}. In \cite[Theorem 2.6.4]{And95}, Ando constructed a series of $\star$-isomorphisms to modify the coordinates such that $g_p(T)$ can approximate $T$ inductively modulo $\m^r$. On the other hand, our proof constructed a $\star$-isomorphism $u = \CN_F^\infty(T)$ directly so that $g_p^{-1}(T) = \CN_F(T)(T) = T$ after a change of coordinate by $u$. Indeed, properties of this limit of operators (i.e., Lemma \ref{Properties of CN} and Lemma \ref{generalized properties of CN}) are deduced inductively on $r$ modulo $\m^r$ similar to the one in Ando's construction. However, our proof cannot proceed by constructing a series of $\star$-isomorphisms, since $\CN_F^\infty(T)$ is not the infinite composite of $\CN_F(T)$. 
    \end{remark}
    \begin{remark}
        Walker has also observed the relationship between the Coleman norm operator and Ando's criterion \cite[Chapter 5]{Wal08}. In particular, he has deduced that Ando's criterion is equivalent to \eqref{Walker's conclusion} in \cite[Lemma 5.0.5 and (5.0.10)]{Wal08}. However, he did not prove Theorem \ref{Ando Theorem 4} via the Coleman norm operator. 
    \end{remark}
    \begin{remark}
        As mentioned in Remark \ref{Zhu and Ando}, Zhu generalized Theorem \ref{Goal} to apply to arbitrary complete local rings. To apply our proof to arbitrary complete local rings, we need to generalize the Coleman norm operator to such cases. However, recall in the argument around \eqref{def of lambda_p,1}, we need $R$ to be a domain so that we can count the number of roots of $\beta$, and we need $p \neq 0$ in $R$ so that $\beta$ is separable. Thus, there is a question whether the Coleman norm operator can be defined over arbitrary complete local rings. 
    \end{remark}

    \section*{Acknowledgements}
    The author is grateful to Charles Rezk for his insight of connecting Coleman's norm operator and Ando's theorem. Furthermore, the author would like to give sincere gratitude to Yifei Zhu for introducing him this idea. The author would also like to thank Eric Peterson and Robert Burklund for useful discussions and suggestions. The author thank Tongtong Liang for teaching me basic notions on power operations. Thanks also to the organizers and speakers of IWoAT (International Workshop on Algebraic Topology) 2022 for the wonderful lectures on chromatic homotopy theory. 
    \par 
    This work was partly supported by the National Natural Science Foundation of China grant 11701263. 

    \nocite{*}
    \Urlmuskip=0mu plus 1mu\relax
    \bibliographystyle{amsalpha}
    \bibliography{Ref-AC}
    \addcontentsline{toc}{section}{References}

    Email: \href{mailto:11911520@mail.sustech.edu.cn}{11911520@mail.sustech.edu.cn} \par 
    Webpage: \href{http://hongxiang-zhao.github.io}{\nolinkurl{hongxiang-zhao.github.io}}

\end{document}